\documentclass[a4paper]{amsart}

\usepackage{amsmath}
\usepackage{amsthm}
\usepackage{amsfonts}
\usepackage{amssymb}
\usepackage{mathrsfs}
\usepackage{stmaryrd} 
\usepackage[all]{xy}
\usepackage{latexsym}
\usepackage{enumerate}
\usepackage{array}
\usepackage{url}
\usepackage{tikz}
\usetikzlibrary{arrows}
\usetikzlibrary{decorations.markings}

\usepackage[lmargin=2.5cm,rmargin=2.5cm,voffset=1cm,footskip=0cm,textheight=23cm]{geometry}

\theoremstyle{plain}
\newtheorem{lemma}{Lemma}[section]
\newtheorem{proposition}[lemma]{Proposition}
\newtheorem{theo}[lemma]{Theorem}
\newtheorem{coro}[lemma]{Corollary}
\theoremstyle{remark}
\newtheorem{rem}[lemma]{Remark}
\newtheorem{notation}[lemma]{Notation}
\theoremstyle{definition}
\newtheorem{definition}[lemma]{Definition}
\newtheorem{ex}[lemma]{Example}

\def\id{\mathop{\rm Id}\nolimits}
\def\inc{\mathop{\rm inc}\nolimits}
\def\K{\mathcal{K}}

\def\Ima{\mathop{\rm Im}\nolimits} 
\def\sgn{\mathop{\rm sgn}\nolimits}
\def\N{\mathbb{N}}
\def\Z{\mathbb{Z}}
\def\tq{\;/\;}
\def\ch{\mathop{\rm Ch}\nolimits}
\def\Ab{\mathcal{A}b}
\def\colim{\mathop{\rm colim}}
\def\coker{\mathop{\rm coker}\nolimits}
\def\hocolim{\mathop{\rm hocolim}}
\def\nhhocolim{\mathop{\underline{\rm hocolim}}}

\def\Aut{\text{Aut}}
\def\ch{\mathop{\rm Ch}\nolimits}

\newcommand\F{\mathcal{F}}
\newcommand\G{\mathcal{G}}
\newcommand{\op}{\textnormal{op}}
\newcommand\rk{\textnormal{rk}}
\newcommand\pr{\textnormal{pr}}
\newcommand\bij{\textnormal{Bij}}
\def\set#1{\llbracket #1 \rrbracket}
\newcommand\g{{\bf g}}

\title[Homology of posets with functor coefficients and Khovanov homology]{Homology of posets with functor coefficients and its relation to Khovanov homology of knots}

\author{Nicol\'as Cianci}
\address{Facultad de Ciencias Exactas y Naturales, Universidad Nacional de Cuyo, Mendoza, Argentina.}
\email{nicocian@gmail.com}

\author{Miguel Ottina}
\address{Facultad de Ciencias Exactas y Naturales, Universidad Nacional de Cuyo, Mendoza, Argentina.}
\email{miguelottina@gmail.com}

\subjclass[2010]{Primary: 55N35, 06A11. Secondary: 57M27.}

\keywords{Homology groups. Khovanov homology. Poset. Alexandroff space.}

\thanks{This research was partially supported by grant and M044 (2016--2018) of SeCTyP, UNCuyo. The first author was also partially supported by a CONICET doctoral fellowship.}

\begin{document}

\begin{abstract}
We study homology groups of posets with functor coefficients and apply our results to give a novel approach to study Khovanov homology of knots and related homology theories.
\end{abstract}

\maketitle

\section{Introduction}

The homology groups of a poset $P$ with coefficients in an abelian group $A$ can be defined as the simplicial homology groups with coefficients in $A$ of its order complex $\K(P)$, that is, the simplicial complex whose faces are the non-empty chains of $P$. These coincide with the singular homology groups with coefficients in $A$ of the classifying space of $P$ when considered as a small category and also coincide with the homology groups of the constant functor with value $A$ from $P$ to the category $\Ab$ of abelian groups.

The homology groups of a poset $P$ with coefficients in a functor $\F\colon P\to \Ab$ are natural generalizations of the latter and are defined in terms of the left-derived functors of the colimit functor $\text{colim}\colon \Ab^{P}\to \Ab$. They were studied by D. Quillen in \cite{quillen1978homotopy} by means of spectral sequences. Concretely, given an order preserving map $f\colon X\to Y$ between posets and a functor $\F\colon P\to \Ab$, Quillen gives a spectral sequence that converges to the homology groups of $X$ with coefficients in $\F$, whose second page can be expressed in terms of the homology groups of $Y$ with coefficients in certain functors $\mathcal{H}_q\colon Y \to \Ab$, $q\geq 0$, that are induced by the functor $\F$ and the order preserving map $f\colon X\to Y$.

\medskip

In \cite{cianci2017new} we computed the integral homology groups of a poset relative to certain subposets and employed these results to derive a homological spectral sequence that can be used to compute the integral homology groups of a poset. In this article, we extend some of these results to homology of posets with functor coefficients and use them together with Quillen's spectral sequence in explicit computations. In particular, we show how the integral homology groups of a finite poset can be efficiently computed in terms of the homology groups of smaller posets with functor coefficients. In addition, we obtain an alternative proof to a well-known generalization of the Mayer-Vietoris exact sequence for basis-like open covers of posets and a version of the Serre spectral sequence with local coefficients for posets.

Moreover, we apply our results to show that Khovanov homology of knots is a special instance of homology of posets with functor coefficients, giving an alternative and more conceptual proof to a similar result given by Everitt and Turner \cite{everitt2009homology}. As a corollary we obtain different proofs to known results of Khovanov homology. We believe that our approach might lead to new insights into Khovanov homology and similar homology theories.

\section{Preliminaries} \label{sect_prelim}

Recall that an Alexandroff space is a topological space $X$ such that the intersection of an arbitrary collection of open subsets of $X$ is open. Thus, if $X$ is an Alexandroff space, for every element $x\in X$ there exists a minimal open set containing $x$, namely the intersection of every neighbourhood of $x$, which will be denoted by $U^{X}_x$, or simply by $U_x$ if the space $X$ is understood and there is no risk of confusion. It is well known that the relation $\leq$ defined on $X$ by $x\leq y$ if and only if $U_x\subseteq U_y$ (or equivalently, if and only if $x\in U_y$) is a preorder, and that this preorder is an order if and only if $X$ is a $T_0$--space. Moreover, every preordered set $(P,\leq)$ can be regarded as an Alexandroff space with the topology given by the down sets of $P$. These (mutually inverse) constructions establish a bijective correspondence between Alexandroff spaces and preordered sets, and between Alexandroff $T_0$--spaces and partially ordered sets. Thus every Alexandroff $T_0$--space will be considered as a poset endowed with the partial order $\leq$ defined above.

If $X$ is an Alexandroff $T_0$--space, the set $U_x$ is just the set $\{y\in X:y\leq x\}$. For every $x\in X$ we define the sets 
\begin{itemize}
\item $F^{X}_x=\{y\in X:y\geq x\}$
\item $C^{X}_x=U^{X}_x\cup F^{X}_x$
\item $\hat{U}^{X}_x=U^{X}_x-\{x\}$
\item $\hat{F}^{X}_x=F^{X}_x-\{x\}$
\item $\hat{C}^{X}_x=C^{X}_x-\{x\}$.
\end{itemize}
As before, if there is no risk of confusion the superscript $X$ will be omitted and these sets will be written simply as $F_x$, $C_x$, $\hat{U}_x$, $\hat{F}_x$ and $\hat{C}_x$ respectively.

Let $X$ be a finite $T_0$--space and let $x\in X$. We say that $x$ is a \emph{down beat point} of $X$ if $\widehat{U}_x$ has a maximum and that $x$ is an \emph{up beat point} of $X$ if $\widehat{F}_x$ has a minimum. We say that $x$ is a \emph{beat point} of $X$ if $x$ is either a down beat point of $X$ or an up beat point of $X$.
We say that a finite $T_0$--space $X$ is a \emph{minimal space} if $X$ does not have beat points. A \emph{core} of $X$ is any minimal space homotopically equivalent to $X$. 

In \cite{stong1966finite}, Stong proved that if $x$ is a beat point of $X$ then $X-\{x\}$ is a strong deformation retract of $X$. Hence, every finite $T_0$--space has a core obtained by successively removing its beat points. Moreover, Stong proved that two finite $T_0$--spaces are homotopy equivalent if and only if their cores are homeomorphic. In particular, the core of a finite $T_0$--space is unique up to homeomorphism.
Using these results, it is easy to see that if there exists $x\in X$ such that $X=U_x$ then $X$ is contractible.

Given an Alexandroff $T_0$--space $X$, the finite non-empty chains of $X$ define a simplicial complex $\K(X)$ called the \emph{order complex} of $X$. In \cite{mccord1966singular}, McCord proved that there is a weak homotopy equivalence from the geometric realization $|\K(X)|$ of $\K(X)$ to $X$.
In particular the singular homology groups of a finite topological space $X$ can be computed as the simplicial homology groups of $\K(X)$.

In \cite{cianci2017new} we developed a spectral sequence that converges to the homology groups of an Alexandroff $T_0$--space $X$ and gave a complete description of the differential morphisms. Under additional hypotheses, the first page of this spectral sequence reduces to a chain complex which allows us to obtain the homology groups of finite posets using a few simple computations.

When considered as a poset, an Alexandroff $T_0$--space $X$ can be regarded as a small category in the standard way, that is, the objects of the category are the elements of $X$ and for $x,y\in X$ there is a (unique) morphism from $x$ to $y$ if and only if $x\leq y$. 
Given a (covariant) functor $\F\colon X\to\Ab$ from $X$ to the category of abelian groups, the homology groups of $X$ with coefficients $\F$ (denoted by $H_n(X;\F)$, $n\geq 0$) are defined as the homology groups of the chain complex $C(X;\F)$ given by \[C_n(X;\F)=\bigoplus_{\sigma \in \ch_n(X)} \F(\min \sigma)\]
where $\ch_n(X)$ denotes the set of $n$-chains of $X$ (that is, chains of $X$ of cardinality $n+1$) and with the differentials defined in the usual way \cite{quillen1978homotopy}. Note that when $\F$ is the constant functor $\Z$, these homology groups are the integral homology groups of $\K(X)$. The functors $L_n\colon \Ab^X \to \Ab$, $n\geq 0$, given by $L_n(\F)=H_n(X;\F)$ are the left-derived functors of the colimit functor \cite{gabriel1967calculus}.

\section{Relative homology with functor coefficients}

Let $X$ be an Alexandroff $T_0$--space and let $\F\colon X\to \Ab$ be a (covariant) functor.
If $A\subseteq X$ is a subspace, we define the homology groups $H_n(X,A;\F)$, $n\geq 0$, as the homology groups of the chain complex $C(X,A;\F)$ defined as the quotient complex $C(X;\F)/C(A;\F|_A)$. Clearly, there is a long exact sequence
\begin{displaymath}
\xymatrix{\cdots \ar[r] & H_n(A;\F|_A) \ar[r]^-{i_\ast} & H_n(X;\F) \ar[r]^-{j_\ast} & H_n(X,A;\F) \ar[r]^-{\partial} & H_{n-1}(A;\F|_A) \ar[r] & \cdots}
\end{displaymath}
Now, if $\F,\G\colon X\to \Ab$ are (covariant) functors and $T\colon F\Rightarrow G$ is a natural transformation then $T$ induces a morphism of chain complexes $\widetilde{T}\colon C(X;\F)\to C(X;\G)$ defined by $\widetilde{T}_n=\bigoplus\limits_{\sigma \in \ch_n(X)} T_{\min \sigma}$. \label{morphism_induced_by_natural_transformation}
Moreover, if $T_A\colon \F|_A\Rightarrow \G|_A$ is the natural transformation induced by $T$, there is a commutative diagram
\begin{displaymath}
\xymatrix{\cdots \ar[r] & H_n(A;\F|_A) \ar[r]^-{i_\ast} \ar[d]_{((\widetilde{T_A})_n)_\ast} & H_n(X;\F) \ar[r]^-{j_\ast} \ar[d]_{(\widetilde{T}_n)_\ast} & H_n(X,A;\F) \ar[r]^-{\partial}  \ar[d]_{(\widetilde{T}_n|)_\ast} & H_{n-1}(A;\F|_A) \ar[r] \ar[d]_{((\widetilde{T_A})_{n-1})_\ast} & \cdots
\\ \cdots \ar[r] & H_n(A;\G|_A) \ar[r]^-{i_\ast} & H_n(X;\G) \ar[r]^-{j_\ast} & H_n(X,A;\G) \ar[r]^-{\partial} & H_{n-1}(A;\G|_A) \ar[r] & \cdots}
\end{displaymath}

If $x\in X$ it is easy to see that $H_0(U_x;\F|_{U_x})=\colim \F|_{U_x}= \F(x)$ and $H_n(U_x;\F|_{U_x})=0$ for $n\geq 1$. Indeed, if $$\cdots \to P_1 \to P_0 \to \F|_{U_x} \to 0$$
is a projective resolution of $\F|_{U_x}$ in $\Ab^{U_x}$ then 
$$\cdots \to \colim P_1 \to \colim P_0 \to \colim \F|_{U_x} \to 0$$
is exact since $\colim P_i=P_i(x)$ for all $i$.

Therefore, we obtain the following simple result.
\begin{proposition}\label{prop_hom_Ux}
Let $X$ be a poset, let $x\in X$ and let $\F\colon X\to\Ab$ be a functor. Let $\alpha\colon \colim \F|_{\widehat U_x} \to \F(x)$ be the morphism induced by the maps $\F(y\to x)$ for $y<x$.

Then
\begin{displaymath}
H_n(U_x,\widehat U_x;\F|_{U_x})\cong \left\{
\begin{array}{ll}
\coker \alpha & \textnormal{if $n=0$}\\
\ker \alpha & \textnormal{if $n=1$} \\
H_{n-1}(\widehat U_x;\F|_{\widehat U_x}) & \textnormal{if $n\geq 2$} 
\end{array}
\right.
\end{displaymath}
\end{proposition}

\begin{proof}
It is not difficult to prove that the morphism $H_0(\widehat U_x;\F|_{\widehat U_x})\to H_0 (U_x;\F|_{U_x})$ induced by the inclusion is $\alpha$. The result follows.
\end{proof}

\begin{proposition}\label{prop_hom_Fx}
Let $X$ be a poset, let $x\in X$ and let $\F\colon X\to\Ab$ and $\G\colon X\to\Ab$ be functors such that $\F(x)$ is isomorphic to $\G(x)$. Then 
\begin{displaymath}
H_n(F_x,\widehat F_x;\F|_{F_x})\cong H_n(F_x,\widehat F_x;\G|_{F_x})
\end{displaymath}
for all $n\in\Z$.
\end{proposition}

\begin{proof}
Note that the chain complexes $C_\ast(F_x,\widehat F_x;\F|_{F_x})$ and $C_\ast(F_x,\widehat F_x;\G|_{F_x})$ are isomorphic. The result follows.
\end{proof}

From now on, if $X$ is a poset and $G$ is an abelian group, $c_G\colon X\to \Ab$ will denote the constant functor with value $G$ (in objects and the identity morphism in arrows).

\begin{rem} \label{rem_natural_transformation}
Let $X$ be a poset, let $x\in X$ and let $\F\colon X\to\Ab$ be a functor. Let $T\colon c_{\F(x)}\Rightarrow \F|_{F_x}$ be the natural transformation defined by $T_a=\F(x\leq a)$ for $a\in F_x$. Thus, $T$ induces a morphism of chain complexes $\widetilde{T}\colon C(F_x;\F(x))\to C(F_x;\F|_{F_x})$, which restricts to the identity map $C(F_x,\widehat{F}_x;\F(x))\to C(F_x,\widehat{F}_x;\F|_{F_x})$. Hence, we obtain isomorphisms
\begin{displaymath}
H_n(F_x,\widehat F_x;\F|_{F_x})\cong H_n(F_x,\widehat F_x;c_{\F(x)})\cong \widetilde H_{n-1}(\widehat{F}_x;\F(x))
\end{displaymath}
for all $n\in\Z$.
\end{rem}

The following result is analogous to Proposition 3.3 of \cite{cianci2017new} for homology of posets with functor coefficients.

\begin{proposition}\label{prop_hom_X_A_Cx}
Let $X$ be a poset, let $\F\colon X\to\Ab$ be a functor and let $A\subseteq X$ such that $X-A$ is an antichain. Then the inclusion maps $i_x\colon (C_x,\widehat{C}_x)\to(X,A)$, $x\in X-A$, induce isomorphisms
\[
\bigoplus_{x\in X-A}\!\!\!H_n(C_x,\widehat C_x;\F|_{C_x})\cong H_n(X,A;\F)
\]
for all $n\in \N_0$.
\end{proposition}

\begin{proof}
Let $n\in \N_0$. Since $X-A$ is an antichain, for every $\sigma\in \ch_n(X)-\ch_n(A)$ there exists a unique $x\in X-A$ such that $x\in \sigma$ or equivalently, such that $\sigma\in \ch_n(C_x)-\ch_n(\widehat{C}_x)$. This correspondence induces a canonical group homomorphism 
$$\phi_n\colon C_n(X,A;\F)\to \bigoplus\limits_{x\in X-A} C_n(C_x,\widehat{C}_x;\F|_{C_x})$$
which is easily seen to be an isomorphism whose inverse is induced by the inclusion maps $(C_x,\widehat{C}_x)\to(X,A)$.
  
Moreover, for every $n\in \N$ we have a commutative diagram
\[
\xymatrix@C=30pt@R=40pt{C_n(X,A;\F)\ar[d]_{\phi_{n}}\ar[rr]^{{d}_n}&&C_{n-1}(X,A;\F)\ar[d]^{\phi_{n-1}}\\
\bigoplus\limits_{x\in X-A}C_n(C_x,\widehat{C}_x;\F|_{C_x})\ar[rr]_{\bigoplus\limits_{x\in X-A}{d}_n^{x}}&&\bigoplus\limits_{x\in X-A}C_{n-1}(C_x,\widehat{C}_x;\F|_{C_x})}
\]
where $d_n\colon C_n(X,A;\F)\to C_{n-1}(X,A;\F)$ is the differential morphism of $C(X,A;\F)$, and where for every $x\in X-A$, the morphism ${d}^{x}_n$ is the differential morphism of $C(C_x,\widehat{C}_x;\F|_{C_x})$ which is just the restriction of ${d}_n$. 
Therefore, the complexes $C(X,A;\F)$ and $\bigoplus\limits_{x\in X-A} C(C_x,\widehat{C}_x;\F|_{C_x})$ are isomorphic and the result follows.
\end{proof}

The following is a simple example of application of the previous results.

\begin{ex} \label{ex_poset_V}
Let ${\bf V}$ be the poset with elements $a$, $b$ and $c$ and with order generated by the relations $a<b$ and $a<c$. Let $\F\colon {\bf V}\to \Ab$ be the functor defined by $\F(a)=\F(c)=\Z$, $\F(b)=0$ and $\F(a<c)(m)=2m$ for all $m\in\Z$, as is shown in the following diagram.
\begin{displaymath}
\begin{tikzpicture}[x=0.8cm,y=0.8cm]
\draw (1,0) node(a){$\bullet$} node[below=1]{$a$};
\draw (0,2) node(b){$\bullet$} node[above=1]{$b$};
\draw (2,2) node(c){$\bullet$} node[above=1]{$c$};
\draw (5,0) node(Fa){$\Z$};
\draw (4,2) node(Fb){$0$};
\draw (6,2) node(Fc){$\Z$};

\draw (a)--(b);
\draw (a)--(c);
\draw[-angle 60] (Fa)--(Fb);
\draw[-angle 60] (Fa)--(Fc) node [midway,right] {\small $\ldotp 2$};

\draw[|-angle 60] (2.5,1)--(3.5,1) node [midway,above] {$\F$};
\end{tikzpicture}
\end{displaymath}

Let $A=\{a,b\}$. By the arguments above we obtain that $H_n(A;\F|_A)=0$ for all $n\in\N_0$ and that 
\begin{displaymath}
H_n({\bf V},A;\F)=H_n(U_c,\widehat U_c;\F|_{U_c})=\left\{
\begin{array}{ll}
H_{n-1}(\widehat U_c;\F|_{\widehat U_c})=H_{n-1}(\widehat U_c;\Z) & \textnormal{if $n\geq 2$} \\
\ker (\F(a<c)) & \textnormal{if $n=1$} \\
\coker (\F(a<c)) & \textnormal{if $n=0$}
\end{array}
\right.
\end{displaymath}
Therefore, $H_n({\bf V};\F)=0$ if $n\geq 1$ and $H_0({\bf V};\F)=\Z_2$.
\end{ex}

Let $f\colon X\to Y$ be a continuous map between posets and let $\F\colon X\to \Ab$ be a functor. For each $q\in\N_0$ there is a functor $\mathcal{H}_q\colon Y\to \Ab$, induced by the map $f$ and the functor $\F$,  which is defined by $\mathcal{H}_q(y)=H_q(f^{-1}(U_y);\F|_{f^{-1}(U_y)})$ and where $\mathcal{H}_q(x<y)$ is the morphism induced by the inclusion map $f^{-1}(U_x)\hookrightarrow f^{-1}(U_y)$. Recall from \cite{quillen1978homotopy} that there is a spectral sequence
\[E^2_{p,q}=H_p(Y;\mathcal{H}_q) \Rightarrow H_{p+q}(X;\F)\ .\]
\label{spectral_sequence}
This spectral sequence can be combined with our tools to concretely compute the homology groups of posets. Consider the following example.

\begin{ex} \label{ex_proj_plane_2}
Let ${\bf P}$ be the poset defined by the following Hasse diagram
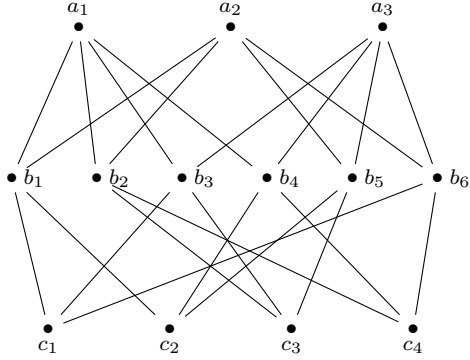
\begin{displaymath}
\begin{tikzpicture}[x=4cm,y=4cm]
	\tikzstyle{every node}=[font=\footnotesize]
	
	\foreach \x in {1,...,3} \draw (0.5*\x,1) node(a\x){$\bullet$} node[above=1]{$a_{\x}$};
	\foreach \x in {1,...,6} \draw (0.28*\x,0.5) node(b\x){$\bullet$} node[right=1]{$b_{\x}$};
	\foreach \x in {1,...,4} \draw (0.4*\x,0) node(c\x){$\bullet$} node[below=1]{$c_{\x}$};
	
	\foreach \x in {1,2,3,4} \draw (a1)--(b\x);
	\foreach \x in {1,2,5,6} \draw (a2)--(b\x);
	\foreach \x in {3,4,5,6} \draw (a3)--(b\x);
	
	\foreach \x in {1,3,6} \draw (c1)--(b\x);
	\foreach \x in {1,4,5} \draw (c2)--(b\x);
	\foreach \x in {2,3,5} \draw (c3)--(b\x);
	\foreach \x in {2,4,6} \draw (c4)--(b\x);
\end{tikzpicture}
\end{displaymath}
The poset ${\bf P}$ is a finite model of the real projective plane \cite{hardie2002nontrivial}, that is, a finite poset such that the geometric realization of its order complex is homotopy equivalent to the real projective plane.
In this example we will show an alternative way to compute the homology groups of this poset using the results of this section.

Let ${\bf V}$ be the poset of example \ref{ex_poset_V} and let $f\colon {\bf P}\to {\bf V}$ be the map that satisfies $f^{-1}(\{a\})=\widehat U^{\bf P}_{a_1}$, $f^{-1}(\{b\})=\{a_1\}$ and $f^{-1}(\{c\})=\{a_2,a_3,b_5,b_6\}$. Note that the map $f$ is continuous.

For each $q\in\N_0$, consider the functor $\mathcal{H}_q\colon {\bf V}\to \Ab$ defined by $\mathcal{H}_q(y)=H_q(f^{-1}(U_y);\Z)$. There is a spectral sequence 
\[E^2_{p,q}=H_p({\bf V};\mathcal{H}_q) \Rightarrow H_{p+q}({\bf P};\Z)\ .\]

Note that $f^{-1}(U_b)=U_{a_1}$ which is contractible and that $f^{-1}(\{a\})$ is a finite model of $S^1$. Note also that $f^{-1}(\{c\})$ is a strong deformation retract of $f^{-1}(U_c)$ since it can be obtained from $f^{-1}(U_c)$ by successively removing beat points. Hence $f^{-1}(U_c)$ is also a finite model of $S^1$. It is not difficult to prove that, up to isomorphism, the morphism $H_1(f^{-1}(\{a\}))\to H_1(f^{-1}(U_c))$ induced by the inclusion map is given by multiplication by 2.

Thus, the functor $\mathcal{H}_1\colon {\bf V}\to \Ab$ is the functor $\F$ of example \ref{ex_poset_V}. Note also that the functor $\mathcal{H}_0$ is the constant functor $\Z$ (with identity maps) and that the functors $\mathcal{H}_q$ for $q\geq 2$ are trivial. Hence, the second page of the spectral sequence above is given by $E^2_{0,0}=\Z$, $E^2_{0,1}=\Z_2$ and $E^2_{p,q}=0$ for $(p,q)\notin \{(0,0),(0,1)\}$. Therefore, $H_0({\bf P};\Z)=\Z$, $H_1({\bf P};\Z)=\Z_2$ and $H_n({\bf P};\Z)=0$ for $n\geq 2$, as expected.
\end{ex}

Example \ref{ex_proj_plane_2} can, in fact, be seen in a more general setting. 

Let $P$ be a finite poset, let $\mathcal{FP}$ be the category of finite posets and let $D\colon P\to\mathcal{FP}$ be a (covariant) functor. In  \cite{fernandez2016homotopy}, Fern\'andez and Minian defined the \emph{non-Hausdorff homotopy colimit of $D$} (denoted by $\nhhocolim D$) as the Grothendieck construction on $D$ and observed that $\hocolim |\K D|$ and $|\K(\nhhocolim D)|$ are homotopy equivalent by Thomason's theorem.

In the following proposition we give a spectral sequence which converges to the homology groups of the non-Hausdorff homotopy colimit of a diagram $D\colon P\to\mathcal{FP}$.

\begin{proposition} \label{prop_hocolim_1}
Let $P$ be a finite poset and let $D\colon P\to\mathcal{FP}$ be a functor. Then there is a spectral sequence 
\[ E^2_{p,q}=H_p(P;\mathcal{H}_q) \Rightarrow H_{p+q}(\nhhocolim D;\Z) \]
where, for each $q\in\Z$, the functor $\mathcal{H}_q\colon P\to \Ab$ is defined as the composition $H_q(D(\_);\Z)$.
\end{proposition}

\begin{proof}
Let $f\colon \nhhocolim D \to P$ be the projection map associated to the Grothendieck construction $\nhhocolim D$ on $D$. Note that, for $y\in P$, $f^{-1}(U_y)=\nhhocolim(D|_{U_y})$, and since the poset $U_y$ has maximum element $y$, from \cite[Corollary 2.5]{fernandez2016homotopy} we obtain that $\nhhocolim(D|_{U_y})$ is homotopy equivalent to $D(y)$ and that the inclusion $i_y\colon D(y)\to \nhhocolim(D|_{U_y})$ is a homotopy equivalence.

Now, given $x,y\in P$ with $x<y$ consider the diagram
\begin{displaymath}
\begin{tikzpicture}[x=4cm,y=2cm]
\draw (0,1) node(Dx){$D(x)$};
\draw (1,1) node(Dy){$D(y)$};
\draw (0,0) node(hocolimDUx){$\nhhocolim(D|_{U_x})$};
\draw (1,0) node(hocolimDUy){$\nhhocolim(D|_{U_y})$};

\draw[-angle 60] (Dx)--(Dy) node [midway,above] {\small $D(x<y)$};
\draw[-angle 60] (Dx)--(hocolimDUx) node [midway,left] {\small $i_x$};
\draw[-angle 60] (Dy)--(hocolimDUy) node [midway,right] {\small $i_y$};
\draw[-angle 60] (hocolimDUx)--(hocolimDUy) node [midway,below] {\small $j$};
\end{tikzpicture}
\end{displaymath}
where the maps $i_x$, $i_y$ and $j$ are the inclusion maps. Since $j \circ i_x\leq i_y \circ D(x<y)$, we obtain that the previous diagram is homotopy commutative. Thus, for all $q\in \N_0$, the functor defined by $y\mapsto H_q(f^{-1}(U_y);\Z)$ coincides with the functor $H_q(D(\_);\Z)$. The result follows.
\end{proof}

Observe that the poset ${\bf P}$ of example \ref{ex_proj_plane_2} can be seen as a non-Hausdorff homotopy colimit in which the map $f\colon {\bf P}\to {\bf V}$ is the projection map.

The previous proposition can be rephrased in terms of basis-like open covers as follows (cf. remark \ref{rem_hocolim}).

\begin{proposition} \label{prop_hocolim_2}
Let $X$ be a finite $T_0$--space and let $\mathcal{V}$ be a basis-like open cover of $X$ with a finite number of elements. We regard $\mathcal{V}$ as a finite poset with the inclusion order. Then there is a spectral sequence 
\[E^2_{p,q}=H_p(\mathcal{V};\mathcal{H}_q) \Rightarrow H_{p+q}(X;\Z)\]
where the functor $\mathcal{H}_q\colon \mathcal{V}\to \Ab$ is defined by $\mathcal{H}_q(V)=H_q(V;\Z)$.
\end{proposition}

\begin{proof}
Since $\mathcal{V}$ is a basis-like open cover, for each $x\in X$ the set $\{V\in\mathcal{V}\tq x\in V\}$ has a minimum. Let $f\colon X\to \mathcal{V}$ be defined by $f(x)=\min\{V\in\mathcal{V}\tq x\in V\}$. Clearly, $f$ is order-preserving and $f^{-1}(U_V)=V$ for all $V\in\mathcal{V}$. The result follows applying the Grothendieck spectral sequence.
\end{proof}

It is not difficult to prove that this spectral sequence is a generalization of the Mayer-Vietoris exact sequence.

\begin{rem}
\label{rem_hocolim}
Note that if $X$ is a finite topological space and $\mathcal{V}$ is a basis-like open cover of $X$ then $X$ is homotopy equivalent to the non-Hausdorff homotopy colimit of the diagram $D\colon \mathcal{V}\to \mathcal{FP}$ defined by $D(V)=V$ in which the morphisms are the inclusion maps. Indeed, let $f\colon X\to \mathcal{V}$ be defined as in the previous proof and let $i\colon X\to \nhhocolim D$ and $r\colon \nhhocolim D \to X$ be defined by $i(x)=(f(x),x)$ and $r(V,x)=x$. Then $i$ and $r$ are order-preserving and satisfy $ri=\id_X$ and $ir\leq \id_{\nhhocolim D}$.

Conversely, if $X$ is the non-Hausdorff homotopy colimit of a diagram $D\colon P\to \mathcal{FP}$, where $P$ is a finite poset, and $f\colon X\to P$ is the projection map then $\{f^{-1}(U_y)\tq y\in P\}$ is a finite basis-like open cover of $X$. This correspondence gives the analogy between propositions \ref{prop_hocolim_1} and \ref{prop_hocolim_2}.
\end{rem}

\subsection*{A Serre spectral sequence for posets}

In this subsection we will apply the spectral sequence given in page \pageref{spectral_sequence} to obtain a version of the Serre spectral sequence with local coefficients for posets.

Recall that a local coefficient system on a topological space $B$ is a functor $\Pi(B)\to \Ab$ where $\Pi(B)$ is the fundamental groupoid of $B$. If $B$ is a poset, then $\Pi(B)$ can be canonically identified with the localization of $B$ with respect to all its morphisms (\cite{cianci2019coverings}) and therefore every morphism-inverting functor $B\to \Ab$ can be factored uniquely as $B\to \Pi(B)\to \Ab$. Hence there is a one-to-one correspondence between morphism-inverting functors from a poset $B$ to $\Ab$ and local coefficient systems on $B$.

Let $B$ be a poset, let $b_0\in B$ and let $\pi=\pi_1(B,b_0)$. Let $\rho\colon B\to \Ab$ be a morphism-inverting functor, let $\overline{\rho}\colon \Pi(B)\to \Ab$ be the corresponding local coefficient system and let $A=\rho(b_0)$.
Then $\overline{\rho}$ induces a group homomorphism $\rho'\colon \pi \to \Aut(A)$ and an associated left $\Z[\pi]$--module structure on $A$ defined (on generators of $\Z[\pi]$) by $g.a=\rho'(g)(a)$ for $g\in \pi$ and $a\in A$. As it is customary, we will write $A_{\rho'}$ instead of $A$ to emphasize the dependence of the module structure of $A$ on the functor $\rho'$. 

Let $\widetilde{B}$ be the universal cover of $B$ and regard $C(\widetilde{B};\Z)$ as a right $\Z[\pi]$--module with multiplication induced by the action of $\pi$ on $\widetilde{B}$.
The (simplicial) homology groups of $B$ with local coefficients in the $\Z[\pi]$--module $A_{\rho'}$ are defined as the homology groups of the chain complex
$$C'(B;A_{\rho'})=C(\widetilde{B};\Z)\otimes_{\Z[\pi]} A_{\rho'}\ .$$

Now, let $\phi\colon \Pi(B)\to \pi$ be an equivalence of categories. It is well known that $\widetilde{B}$ is isomorphic to the poset whose underlying set is $B\times\pi$ with $(x,g)\leq (y,h)$ if and only if $x\leq y$ and $h\phi(x\leq y)=g$ (see \cite{cianci2018splitting} for details). With this characterization of $\widetilde{B}$ it is easy to see that $C'(B;A_{\rho'})\cong C(B;\rho)$ and hence the homology groups of $B$ with local coefficients in $A_{\rho'}$ are naturally isomorphic to the homology groups of $B$ with coefficients in the functor $\rho$ (cf. \cite{eilenberg1947homology}).

The next proposition, which follows easily from the previously stated facts, gives a version of the Serre spectral sequence with local coefficients for posets.

\begin{proposition}
  Let $X$ and $B$ be posets and let $G$ be an abelian group. Let $b_0\in B$ and let $f\colon X\to B$ be a continuous function such that the inclusion $f^{-1}(U_b)\hookrightarrow f^{-1}(U_{b'})$ is a weak homotopy equivalence for every $b,b'\in B$ such that $b\leq b'$. 
  
  Let $\mathcal{H}_q\colon B\to \Ab$ be the (morphism-inverting) functor defined by \[\mathcal{H}_q(b)=H_q(f^{-1}(U_b);G)\] for every $b\in B$ and \[\mathcal{H}_q(b\leq b')=H_q(f^{-1}(U_b)\hookrightarrow f^{-1}(U_{b'});G)\] for every $b,b'\in B$ such that $b\leq b'$.
  
  Then there is a spectral sequence 
  \[E^2_{p,q}=H_p(B;A^q_{\rho_q}) \Rightarrow H_{p+q}(X;G)\]
  where $\rho_q\colon \Pi(B)\to \Ab$ is the local coefficient system induced by $\mathcal{H}_q$ and $A^q_{\rho_q}=\rho_q(b_0)$ is the left $\Z[\pi]$--module associated to $\rho_q$ for every $q\geq 0$.
\end{proposition}

\begin{proof}
Follows from the natural isomorphisms $H_p(B;\mathcal{H}_q)\cong H_p(B;A^q_{\rho_q})$ for $p\geq 0$.
\end{proof}

\begin{ex}

Let $X$ be the following poset, which is a finite model of the Klein bottle \cite{cianci2018splitting}:
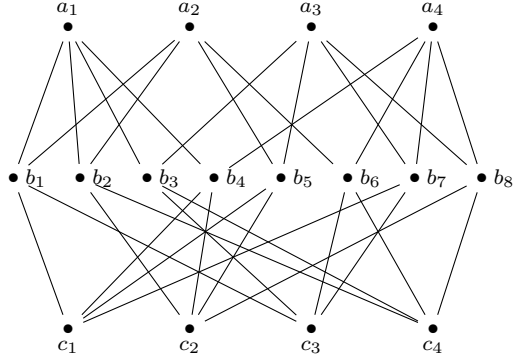
  \begin{center}
    \begin{tikzpicture}[x=4cm,y=4cm]
      \tikzstyle{every node}=[font=\footnotesize]
      \foreach \x in {1,...,4} \draw (0.4*\x,1) node(a\x){$\bullet$} node[above=1]{$a_{\x}$};
      \foreach \x in {1,...,8} \draw (0.22*\x,0.5) node(b\x){$\bullet$} node[right=1]{$b_{\x}$};
      \foreach \x in {1,...,4} \draw (0.4*\x,0) node(c\x){$\bullet$} node[below=1]{$c_{\x}$};
      \foreach \x in {1,2,3,4} \draw (a1)--(b\x);
      \foreach \x in {1,2,5,6} \draw (a2)--(b\x);
      \foreach \x in {3,5,7,8} \draw (a3)--(b\x);
      \foreach \x in {4,6,7,8} \draw (a4)--(b\x);
      \foreach \x in {1,4,5,7} \draw (c1)--(b\x);
      \foreach \x in {2,4,5,8} \draw (c2)--(b\x);
      \foreach \x in {1,3,6,7} \draw (c3)--(b\x);
      \foreach \x in {2,3,6,8} \draw (c4)--(b\x);
    \end{tikzpicture}
  \end{center}
  Let $B$ be the poset
    
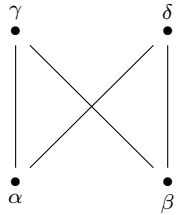
\begin{center}
      \begin{tikzpicture}[x=2cm,y=2cm]
      \tikzstyle{every node}=[font=\footnotesize]
	\draw (0,0) node(alpha){$\bullet$} node[below=1]{$\alpha$};
	\draw (1,0) node(beta){$\bullet$} node[below=1]{$\beta$};
	\draw (0,1) node(gamma){$\bullet$} node[above=1]{$\gamma$};
	\draw (1,1) node(delta){$\bullet$} node[above=1]{$\delta$};
	
	\draw (alpha)--(gamma);
	\draw (alpha)--(delta);
	\draw (beta)--(gamma);
	\draw (beta)--(delta);      
	
    \end{tikzpicture}
  \end{center}
and let $p\colon X\to B$ be the function defined by $p^{-1}(\alpha)=\{b_4,b_5,c_1,c_2\}$, $p^{-1}(\beta)=\{b_3,b_6,c_3,c_4\}$, $p^{-1}(\gamma)=\{a_1,a_2,b_1,b_2\}$ and $p^{-1}(\delta)=\{a_3,a_4,b_7,b_8\}$. It is easy to see that $p^{-1}(\alpha)\hookrightarrow X\xrightarrow{p} B$ is a fiber bundle and hence $p$ is a Hurewicz fibration \cite[Theorem 4.4]{cianci2019classification}. Observe that $p^{-1}(\alpha)$ and $B$ are finite models of $S^{1}$.

Let $\F\colon X\to \Ab$ be the constant $\Z$ functor. It is clear that $\mathcal{H}_0\colon B\to \Ab$ is the constant $\Z$ functor and therefore 
$H_p(B;\mathcal{H}_0)=\Z$ for $p=0,1$ and $H_p(B;\mathcal{H}_0)=0$ for $p\geq 2$.
  
On the other hand, the functor $\mathcal{H}_1\colon B\to \Ab$ is given by the diagram 
  \begin{center}
    \begin{tikzpicture}[x=2cm,y=2cm]
    \tikzstyle{every node}=[font=\footnotesize]
	\draw (0,0) node(alpha){$\Z$};
	\draw (1,0) node(beta){$\Z$};
	\draw (0,1) node(gamma){$\Z$};
	\draw (1,1) node(delta){$\Z$};
	
	\draw[->] (alpha)->(gamma);
	\draw[->] (alpha)->(delta);
	\draw[->] (beta)->(gamma);
	\draw[->] (beta)->(delta);      
	
	\draw (0,0.5) node[left]{1};
	\draw (1,0.5) node[right]{1};
	\draw (0.25,0.25) node[right]{-1};
	\draw (0.75,0.25) node[left]{1};
  \end{tikzpicture}
\end{center}
and a direct computation shows that $H_0(B;\mathcal{H}_1)=\Z_2$ and $H_p(B;\mathcal{H}_1)=0$ for $p\geq 1$.
  
Since $\mathcal{H}_n$ is trivial for $n\geq 2$, it follows that $H_0(X;\Z)=\Z$, $H_1(X;\Z)=\Z\oplus\Z_2$ and $H_n(X;\Z)=0$ for $n\geq 2$. 
\end{ex}

\section{Khovanov homology of knots}

In this section we will apply our results to show that Khovanov homology of knots is a special instance of homology with functor coefficients, giving an alternative and more conceptual proof to a similar result given by Everitt and Turner \cite{everitt2009homology}. 

Let $P$ be a poset with maximum element $1$ and let $\F\colon P\to\Ab$ be a functor.

We recall the following definition from \cite{cianci2017new}.

\begin{definition}\label{def_quasicellular}
Let $X$ be a locally finite $T_0$--space (or poset). We say that $X$ is \emph{quasicellular} if there exists an order preserving map $\rho\colon X\longrightarrow \N_0$, which will be called \emph{quasicellular morphism for $X$}, such that
\begin{enumerate}[(1)]
\item The set $\{x\in X:\rho(x)=n\}$ is an antichain for every $n\in \N_0$. 
\item For every $x\in X$, the reduced homology of $\widehat{U}_x$ is concentrated in degree $\rho(x)-1$.
\end{enumerate}
\end{definition}

Note that if $X$ is a quasicellular poset, $\rho$ is a quasicellular morphism for $X$ and $x,y\in X$, then $x<y$ implies that $\rho(x)<\rho(y)$. Note also that if $X$ is a quasicellular poset, then $\widetilde{H}_{\rho(x)-1}(\widehat U_x)$ is finitely generated for all $x\in X$.

If $X$ is a poset, $X^\op$ will denote the poset $X$ with the inverse order and will be called the \emph{opposite} poset of $X$.

\begin{theo} \label{theo_homology_quasicellular_opposite}
Let $P$ be a poset such that $P^\op$ is a quasicellular poset with quasicellular morphism $\rho\colon P^\op\to \N_0$ satisfying that $\widetilde{H}_{\rho(x)-1}(\widehat U_x^{P^\op})$ is a free abelian group for all $x\in P$. For each $n\in\Z$ let $D_n=\{x\in X \tq \rho(x)=n\}=\rho^{-1}(\{n\}\cap\N_0)$.

Let $\F\colon P\to\Ab$ be a functor. 
Let $\mathcal{C}^{P,\F}=(C^{P,\F}_\ast,d^{P,\F}_\ast)$ be the chain complex defined by 
\begin{displaymath}
\displaystyle C^{P,\F}_n=\bigoplus\limits_{x\in D_n}\!\!\!\widetilde{H}_{n-1}(\widehat{F}_x;\F(x))
\end{displaymath}
for all $n\in\Z$ with differentials $d^{P,\F}_n\colon C^{P,\F}_n\to C^{P,\F}_{n-1}$ defined by
\begin{displaymath}
d^{P,\F}_n\left(\left(\left[\sum\limits_{i=1}^{l_x}a^x_i s^x_i\right]\right)_{x\in D_n}\right)=\left(\left[\sum\limits_{x\in D_n}\sum\limits_{s^x_i\ni y}\F(x\leq y)(a^x_i)\ldotp(s^x_i-\{y\})\right]\right)_{y\in D_{n-1}}
\end{displaymath}
where for every $x\in D_n$, $l_x\in \N$, and for every $i=\{1,\dots,l_x\}$, $a^x_i\in \F(x)$ and $s^x_i\in \ch_{n-1}(\widehat{F}_x)$.

Then $H_n(P;\F)\cong H_n(\mathcal{C}^{P,\F})$ for all $n\in \Z$.
\end{theo}

\begin{proof}
For each $m\in\N_0$ let $P_m=\rho^{-1}(\{0,\ldots,m\})$ with the order induced by the one in $P$. Applying propositions \ref{prop_hom_X_A_Cx} and \ref{prop_hom_Fx} we obtain that
\begin{displaymath}
H_n(P_m,P_{m-1};\F|) \cong \bigoplus_{x\in D_m}H_n(F_x,\widehat F_x,\F|)\cong \bigoplus_{x\in D_m}\widetilde H_{n-1}(\widehat F_x,\F(x))
\end{displaymath}
It follows that if $l,m,n\in \N_0$ are such that either $l<m<n$ or $n<l<m$ then the inclusion $P_l\to P_m$ induces an isomorphism $H_n(P_l;\F|)\to H_n(P_m;\F|)$. In particular, we obtain that $H_n(P_m;\F|)\cong H_n(P_0;\F|)=0$ if $n>m$. On the other hand, it is easy to check that $H_n(P;\F)=\colim\limits_{m\in\N_0} H_n(P_m;\F|)$ for all $n\in\N_0$ since any cycle is contained in $P_m$ for some $m$.

Therefore, either working as in the case of cellular homology of CW--complexes or applying a spectral sequence argument one obtains a chain complex $C'=(C'_\ast,d'_\ast)$ defined by $C'_n=H_n(P_n,P_{n-1};\F|)$ for all $n\in\Z$ with differentials $d_n\colon C'_n\to C'_{n-1}$ defined by the composition $H_n(P_n,P_{n-1};\F|)\to H_{n-1}(P_{n-1};\F|)\to H_{n-1}(P_{n-1},P_{n-2};\F|)$ such that $H_n(P;\F)\cong H_n(C')$ for all $n\in \Z$.

Since $C'_n=H_n(P_n,P_{n-1};\F|)\cong \bigoplus\limits_{x\in D_m}\widetilde H_{n-1}(\widehat F_x,\F(x))$, it remains to prove that under this isomorphisms the differentials are given by the formula stated in the theorem.

Let $n\in \N$. Consider the following commutative diagram
\begin{displaymath}
\xymatrix@C=38pt@R=40pt{H_n(P_n,P_{n-1};\F|_{P_n})\ar@/^2pc/[rr]^-{d'_{n}}\ar[r]^-{\partial}
& H_{n-1}(P_{n-1};\F|_{P_{n-1}})\ar[r]^-{j_*} 
& H_{n-1}(P_{n-1},P_{n-2};\F|_{P_{n-1}})\ar[d]_{(\phi_{n-1})_*}^\cong
\\ \bigoplus\limits_{x\in D_n}\!\!\!H_n(F_x,\widehat{F}_x;\F|_{F_x})\ar[u]^-{i_*}_\cong \ar[r] 
& \bigoplus\limits_{x\in D_n}\!\!\!H_{n-1}(\widehat{F}_x;\F|_{\widehat{F}_x}) \ar[u]^-{i'_*}
& \bigoplus\limits_{y\in D_{n-1}}\!\!\!\!\!\!H_{n-1}(F_y,\widehat{F}_y;\F|_{F_y})
\\ \bigoplus\limits_{x\in D_n}\!\!\!H_n(F_x,\widehat{F}_x;\F(x)) \ar[u]^{\bigoplus\limits_{x\in D_{n}}\!\!\!\!\tau_x}_\cong \ar[r]_\cong
& \bigoplus\limits_{x\in D_n}\!\!\!\widetilde{H}_{n-1}(\widehat{F}_x;\F(x)) \ar[rd]_-{d_n} \ar[u]^{\overline{\tau}_n=\bigoplus\limits_{x\in D_{n}}\!\!\!\!\tau_x}
& \bigoplus\limits_{y\in D_{n-1}}\!\!\!\widetilde{H}_{n-1}(F_y,\widehat{F}_y;\F(y)) \ar[u]^{\overline{\tau}_{n-1}=\bigoplus\limits_{y\in D_{n-1}}\!\!\!\!\!\!\tau_y}_\cong \ar[d]_{\overline{\partial}=\bigoplus\limits_{y\in D_{n-1}}\!\!\!\!\!\!\partial^y}^\cong
\\ 
&
& \bigoplus\limits_{y\in D_{n-1}}\!\!\!\widetilde{H}_{n-2}(\widehat{F}_y;\F(y))}
\end{displaymath}
where
\begin{itemize}
\item the maps $i_\ast$ and $i'_\ast$ are induced by the inclusion maps $(F_x,\widehat{F}_x)\to (P_n,P_{n-1})$, $x\in D_n$.
\item the map $\phi_{n-1}$ is defined as in the proof of \ref{prop_hom_X_A_Cx}
\item for $y\in D_{n-1}$ the map $\partial^y\colon H_{n-1}(F_y,\widehat{F}_y;\F(y))\to H_{n-1}(F_y,\widehat{F}_y;\F(y))$ is defined by
$\partial^y\left(\left[\sum\limits_{i=1}^{l}a_i s_i\right]\right)=\left[\sum\limits_{i=1}^{l}a_i (s_i-\{y\})\right]$ where $l\in\N$, and $a_i\in \F(y)$ and $s_i\in\ch_{n-1}(F_y,\widehat{F}_y)$ for $i\in\{1,\ldots,l\}$
\item the maps $\tau_x$ and $\tau'_x$ (for $x\in D_n$) and $\tau_y$ (for $y\in D_{n-1}$) are induced by the natural trasformations $T^z\colon c_{\F(z)}\Rightarrow \F|_{F_z}$, $z\in D_n\cup D_{n-1}$, as in remark \ref{rem_natural_transformation}.
\end{itemize}

Note that the maps labeled with $\cong$ are isomorphisms.

Let $\sigma=([\sigma_x])_{x\in D_n}$ with $[\sigma_x]\in \widetilde H_{n-1}(\widehat{F}_x,\F(x))$ for each $x\in D_n$. For each $x\in D_n$ we write $\sigma_x=\sum\limits_{i=1}^{l_x}a_i^x s^x_i$ with $l_x\in \N$, $a_i^x\in \F(x)$ and $s^x_i$ an $(n-1)$--chain of $\widehat{F}_x$ for every $i\in \{1,\dots,l_x\}$. We have that

\begin{align*}
d_n\left(\left(\left[\sum\limits_{i=1}^{l_x}a^x_i s^x_i\right]\right)_{x\in D_n}\right) & = (\overline{\partial}(\overline{\tau}_{n-1})^{-1} (\phi_{n-1})_\ast j_\ast i'_\ast \overline{\tau}_n)\left(\left(\left[\sum\limits_{i=1}^{l_x}a^x_i s^x_i\right]\right)_{x\in D_n}\right) = \\
&= (\overline{\partial}(\overline{\tau}_{n-1})^{-1} (\phi_{n-1})_\ast j_\ast i'_\ast ) \left(\left(\left[\sum\limits_{i=1}^{l_x}\F(x\leq \min(s^x_i))(a^x_i)\ldotp s^x_i\right]\right)_{x\in D_n}\right) = \\ 
&= (\overline{\partial}(\overline{\tau}_{n-1})^{-1} (\phi_{n-1})_\ast) \left(\left[\sum_{x\in D_n}\sum\limits_{i=1}^{l_x}\F(x\leq \min(s^x_i))(a^x_i)\ldotp s^x_i\right]\right) = \\ 
& = \left(\left[\sum\limits_{x\in D_n}\sum\limits_{s^x_i/\min(s^x_i)=y}\F(x\leq y)(a^x_i)\ldotp(s^x_i-\{y\})\right]\right)_{y\in D_{n-1}} \ .
\end{align*}
\end{proof}

Recall that a locally finite poset $P$ is graded if it is equipped with a rank function $\rk\colon P\to \N_0$ which is order-preserving and satisfies that if $x\prec y$ is a cover relation of $P$ then $\rk(y)=\rk(x)+1$. In particular, any finite boolean lattice is graded since it is isomorphic to the power set of a finite set. 

Recall that if $X$ is a finite poset, the \emph{height} of $X$, denoted $h(X)$, is one fewer than the maximum cardinality of a chain of $X$. Note that $h(X)=\dim (\K(X))$.

Now, let $L$ be a finite boolean lattice. We define the \emph{rank function} $\rk\colon L\to\N_0$ by $\rk(x)=h(\widehat U_x)$. If $x\in L$, note that $\rk(x)=1$ if and only if $x$ is an atom of $L$ and that $\rk(x)=l$ if and only if the unique expression of $x$ as a join of (distinct) atoms contains $l$ atoms. We also define the \emph{rank} of $L$ as $\rk(L)=\rk(\max(L))$. Note that there exists an isomorphism (of posets) between $L$ and the power set of the finite set $\{1,\ldots,\rk(L)\}$. Equivalently, we can consider the set $\{0,1\}$ with the partial order induced by the relation $0\leq 1$ and thus there exists an isomorphism (of posets) $\phi\colon L\to \{0,1\}^{\rk(L)}$ where $\{0,1\}^{\rk(L)}$ is endowed with the product order. Observe that such an isomorphism $\phi\colon L\to \{0,1\}^{\rk(L)}$ determines (and is determined by) a total ordering on the atoms of $L$.

From now on, we will choose a total ordering on the atoms of $L$ to obtain a fixed isomorphism $\phi\colon L\to \{0,1\}^{\rk(L)}$. For $j\in \{1,\ldots,\rk(L)\}$ let $\pr_j\colon \{0,1\}^{\rk(L)}\to \{0,1\}$ be the projection map to the $j$--th coordinate and let $\phi_j=\pr_j\phi$.

Observe that if $x\prec y$ is a cover relation of $L$, then there exists a unique $m_{xy}\in\{1,\ldots\rk(L)\}$ such that $\phi_{m_{xy}}(y)=1$, $\phi_{m_{xy}}(x)=0$ and $\pr_j(y)=\pr_j(x)$ for $j\in \{1,\ldots\rk(L)\} - \{m_{xy}\}$. We define the \emph{sign} of the cover relation $x\prec y$ as $\varepsilon(x\prec y)=\prod\limits_{j=1}^{m_{xy}}(-1)^{\phi_j(x)}$, that is, $\varepsilon(x\prec y)=-1$ if the cardinality of the set $\{j\in\{1,\ldots,m_{xy}\}\tq \phi_j(x)=1\}$ is odd and $\varepsilon(x\prec y)=1$ otherwise.

From now on, for $n\in\N$ let $\set n=\{1,\ldots,n\}$ and let $\mathcal{S}_n$ be the symmetric group of degree $n$. Also, if $n\in \N$, for each subset $A\subseteq \set n$ let $\eta_A\colon \set{\#A}\to A$ be the only bijective and order-preserving map. If $X$ and $Y$ are sets, we define $\bij(X,Y)$ as the set of bijective maps from $X$ to $Y$.

Let $n\in\N$ such that $n\geq 2$ and let $P$ be the poset of non-empty subsets of $\set{n}$ ordered by inclusion. In the following lemma, for each $A\subseteq \set{n}$ we will find an explicit generator for the relative homology group $H_{\#A-1}(U_A,\widehat U_A;\Z)\cong \widetilde H_{\#A-2}(\widehat U_A)\cong \Z$ (note that $\K(U_A)$ is the barycentric subdivision of the $(\#A -1)$--simplex and $\K(\widehat U_A)$ is its boundary). We will also compute certain cellular-type boundary maps which will be used later to show that the usual sign convention in Khovanov's cubical complex is induced by the boundary maps of the cellular chain complex of $\K(\set{n})$ (after a suitable selection of generators).

To this end, observe that for $A\subseteq \set{n}$, if $S_1\subsetneq S_2\subsetneq \ldots\subsetneq S_{\#A}$ is an $(\#A-1)$--chain of $(U_A,\widehat U_A)$ then $S_{\#A}=A$ and $\# S_j=j$ for all $j\in \set{\#A}$. It follows that there is a bijection $\psi\colon \bij(\set{\#A},A) \to \ch_{\#A-1}(U_A,\widehat U_A)$ given by $\psi(\alpha)=\{\alpha(\set{1}),\alpha(\set{2}),\ldots,\alpha(\set{\#A})\}$.

Note also that there is a bijection $\nu\colon \mathcal{S}_{\#A}\to \bij(\set{\#A},A)$ defined by $\nu(\sigma)=\eta_A\circ \sigma$.

\begin{lemma} \label{lemma_generators}
Let $n\geq 2$ and let $P=\mathcal{P}_{\neq\varnothing}(\{1,\ldots,n\})$. 

(1)  Let $A\in P$, let $k=\#A$ and let $\g_A\in C_{k-1}(U_A,\widehat U_A;\Z)$ be defined by $\g_A=\sum\limits_{\sigma\in\mathcal{S}_k}\sgn(\sigma)\psi(\eta_A \circ \sigma)$.
Then $\g_A$ is a generator of $H_{k-1}(U_A,\widehat U_A;\Z)\cong \widetilde H_{k-2}(\widehat U_A)\cong \Z$.

(2) Let $A\in P$ with $\#A\geq 2$ and let $k=\#A$. Let $a\in A$ and let $h=\eta_A^{-1}(a)$. Let $d\colon \widetilde H_{k-1}(U_A,\widehat U_A;\Z)\to \widetilde H_{k-2}(U_{A-\{a\}},\widehat U_{A-\{a\}};\Z)$ be the group homomorphism defined by
$$d\left(\left[\sum\limits_{i=1}^{l}n_i \zeta_i\right]\right)=\left[\sum\limits_{\zeta_i\ni A-\{a\}}n_i(\zeta_i-\{A\})\right]$$
where $l\in \N$, and for $i\in\{1,\ldots,l\}$, $n_i\in\Z$ and $\zeta_i\in\ch_{k-1}(U_A,\widehat U_A)$.
  
Then $d([\g_A])=(-1)^{k-h}[\g_{A-\{a\}}]$, and thus, the morphism $\tilde d\colon \Z\to\Z$ induced by $d$ under the isomorphisms of (1) that take $1\in\Z$ to $\g_A$ and $\g_{A-\{a\}}$ respectively, is the isomorphism $(-1)^{k-h}\id_\Z$.  
\end{lemma}

\begin{proof}
(1) 
First of all, we wish to show that $\g_S$ is a $(k-1)$--cycle of $(U_S,\widehat U_S)$. It suffices to show that $(p_{\xi}\circ \partial)(\g_S)=0$ for every $\xi\in \ch_{k-2}(U_S,\widehat U_S)$, where $\partial\colon C_{k-1}(U_S,\widehat U_S;\Z)\to C_{k-2}(U_S,\widehat U_S;\Z)$ denotes the differential of the chain complex $C(U_S,\widehat U_S;\Z)$ and $p_\xi$ denotes the projection of $C_{k-2}(U_S,\widehat U_S;\Z)$ to the direct summand generated by the $(k-2)$--chain $\xi$.

Let $\xi=\{\xi_1,\ldots,\xi_{k-1}\}$ be a $(k-2)$--chain of $(U_S,\widehat U_S)$ and let $\xi_0=\varnothing$. 
By cardinality, there exists exactly one natural number $j\in\{1,\ldots,k-1\}$ such that $\#(\xi_j-\xi_{j-1})=2$, while $\#(\xi_l-\xi_{l-1})=1$ for every $l\neq j$. Suppose that $\xi_j-\xi_{j-1}=\{b,c\}$ with $b<c$. 

It follows that there are exactly two $(k-1)$--chains of $(U_A,\widehat U_A)$ that contain $\xi$. These are the chains $\zeta=\xi\cup\{\xi_{j-1}\cup\{b\}\}$ and $\zeta'=\xi\cup\{\xi_{j-1}\cup\{c\}\}$. Let $\sigma=\psi^{-1}(\zeta)$ and $\sigma'=\psi^{-1}(\zeta')$. Let $T\colon A\to A$ be the transposition that interchanges the elements $b$ and $c$. Then $\sigma'=T\circ \sigma$. Hence $\sgn(\sigma)=-\sgn(\sigma')$ and it follows that $(p_{\xi}\circ \partial)(\g_S)=(-1)^{j+1}(\sgn(\sigma)+\sgn(\sigma'))\xi=0$.

Thus $\g_A$ is a $(k-1)$--cycle of $(U_A,\widehat U_A)$. Since $C_k(U_A,\widehat U_A;\Z)$ is trivial, we obtain that $H_{k-1}(U_S,\widehat U_S;\Z) \cong \ker \partial$ and hence $\g_A$ represents a non-trivial element of $H_{k-1}(U_A,\widehat U_A;\Z)$. Now, let $g$ be a generator of $\ker\partial$. Then $g$ is a sum of $(k-1)$--chains of $(U_A,\widehat U_A)$ and thus it can be written as $g=\sum\limits_{\sigma\in\mathcal{S}_k}a_\sigma\psi(\eta_A\circ\sigma)$ where $a_\sigma\in\Z$ for every $\sigma\in\mathcal{S}_k$. 

Now, there exists $m\in\Z-\{0\}$ such that $\g_A=mg$ and therefore $\sgn(\sigma)=m a_\sigma$ for every $\sigma\in\Aut(S)$. It follows that $m\in\{1,-1\}$ and thus $\g_A$ is a generator of $\ker\partial$.

(2) 
Let $\sigma \in\mathcal{S}_k$. Note that  $A-\{a\} \in \psi(\eta_A\circ\sigma)$ if and only if $(\eta_A\circ \sigma)(k)=a$, which occurs if and only if $\sigma(k)=h$.

Now, let $\sigma \in\mathcal{S}_k$ such that $A-\{a\} \in \psi_A(\eta_A\circ\sigma)$. Then $\psi_A(\eta_A\circ\sigma)-\{A\} \in \ch_{k-2}(U_{A-\{a\}},\widehat U_{A-\{a\}})$ and $\psi_{A-\{a\}}(\eta_A\circ\sigma|_{\set{k-1}}^{A-{a}})=\psi_A(\eta_A\circ\sigma)-\{A\}$.

Let $\gamma\in\mathcal{S}_k$ be the cyclic permutation $(k\  {k-1}\ \cdots\  h)$ and let $\sigma \in\mathcal{S}_k$ such that $\sigma(k)=h$.
Note that $(\gamma\circ\sigma)(k)=k$. Let $\gamma|\colon \set{k}-\{h\}\to \set{k-1}$ and $\sigma|\colon \set{k-1}\to \set{k}-\{h\}$ be restrictions of $\gamma$ and $\sigma$ respectively and let $\tau_\sigma=\gamma|\circ\sigma|\colon \set{k-1}\to \set{k-1}$. It follows that $\tau_\sigma\in S_{k-1}$ and that $\sgn(\tau_\sigma)=\sgn(\gamma\circ\sigma)=(-1)^{k-h}\sgn(\sigma)$. 
Observe that there is a bijection $\{\sigma \in\mathcal{S}_k \tq \sigma(k)=h\} \to \mathcal{S}_{k-1}$ given by $\sigma \mapsto \tau_\sigma$.

Now let $\eta_A|\colon \set{k}-\{h\} \to A-\{a\}$ be the restriction of $\eta_A$. Since $\gamma|^{-1}$ is order-preserving it follows that $\eta_A|\circ \gamma|^{-1}\colon \set{k-1}\to A-\{a\}$ is also order-preserving and hence $\eta_A|\circ \gamma|^{-1}=\eta_{A-\{a\}}$. Thus $\eta_A|=\eta_{A-\{a\}}\circ \gamma|$.

Therefore,

$$d([\g_A])=d\left(\left[\sum\limits_{\sigma\in\mathcal{S}_k}\sgn(\sigma)\psi(\eta_A\circ\sigma)\right]\right)=\left[\sum\limits_{\substack{\sigma\in\mathcal{S}_k\\ \sigma(k)=h}}\sgn(\sigma)(\psi_A(\eta_A\circ\sigma)-\{A\})\right]=$$
$$ = \left[\sum\limits_{\substack{\sigma\in\mathcal{S}_k\\ \sigma(k)=h}}\sgn(\sigma) \psi_{A-\{a\}}(\eta_A|\circ\sigma|)\right] = \left[\sum\limits_{\substack{\sigma\in\mathcal{S}_k\\ \sigma(k)=h}}\sgn(\sigma) \psi_{A-\{a\}}(\eta_{A-\{a\}}\circ \gamma| \circ\sigma|)\right] = $$
$$= \left[\sum\limits_{\tau\in\mathcal{S}_{k-1}}(-1)^{k-h} \sgn(\tau) \psi_{A-\{a\}}(\eta_{A-\{a\}}\circ \tau)\right] = (-1)^{k-h}[\g_{A-\{a\}}].$$
\end{proof}

Let $n=\rk(L)$. Let $\omega\colon \{0,1\}^{\rk(L)}\to \mathcal{P}_{\neq\varnothing}(\{1,\ldots,n\})$ be the order-reversing bijection given by $\omega(b_1,\ldots,b_n)=\{j\in\{1,\ldots,n\} \tq b_{n+1-j}=0\}$ and let $\theta=\omega\phi$. Observe that for each $x\in L-\{1_L\}$ the map $\theta$ induces isomorphisms $\widetilde H_{l}(F_x,\widehat F_x;\F(x))\to \widetilde H_{l}(U_{\theta(x)},\widehat U_{\theta(x)};\F(x))$ and $\widetilde H_{l}(\widehat F_x;\F(x))\to \widetilde H_{l}(\widehat U_{\theta(x)};\F(x))$ for all $l\in\Z$, which will all be denoted by $\theta_\ast$.
Let $\tau\colon L-\{1_L\}\to \Z$ be defined by $\tau(x)=\#\theta(x)+\sum_{j\in\theta(x)}j +1$ and let $\varsigma\colon L-\{1_L\}\to \{1,-1\}$ be defined by $\varsigma(x)=(-1)^{\tau(x)}$.

\begin{lemma} \label{lemma_sign_comparison}
Let $L$ be a finite boolean lattice and let $n=\rk(L)$. Let $\tau\colon L\to \Z$ be defined by $\tau(z)=\sum_{j\notin\theta(z)}(j+n)$ and let $\varsigma\colon L\to \{1,-1\}$ be defined by $\varsigma(z)=(-1)^{\tau(x)}$.

Let $x,y\in L$ such that $x\prec y$. Let $a$ be the only element of $\theta(x)-\theta(y)$ and let $h=\eta_{\theta(x)}^{-1}(a)$.

Then $\varsigma(y)\varepsilon(x\prec y)= \varsigma(x) (-1)^{\#\theta(x)-h}$.
\end{lemma}

\begin{proof}
Let $q=\#\{j\in\{1,\ldots,m_{xy}\}\tq \phi_j(x)=1\}$. Recall that $\varepsilon(x\prec y)=(-1)^q$. Thus, it suffices to prove that $\tau(y)+q \equiv \tau(x)+\#\theta(x)-h \ \textnormal{(mod 2)}$.

Note that $\theta(x)-\theta(y)=\{n+1-m_{xy}\}$. Hence $a=n+1-m_{xy}$. Note also that $\#\{j\in\{1,\ldots,a\}\tq j\in \theta(x)\} = \eta_{\theta(x)}^{-1}(a)$ since $\eta_{\theta(x)}\colon \set{\#\theta(x)}\to \theta(x)$ is order-preserving.

Observe that
\begin{align*}
q & = \#\{j\in \set{m_{xy}} \tq \phi_j(x)=1\}=\#\{j\in\{n+1-m_{xy},\ldots,n\}\tq \phi_{n+1-j}(x)=1\} = \\
& = \#\{j\in\{n+1-m_{xy},\ldots,n\}\tq j\notin \theta(x)\}= \#\theta(x)^c-\#\{j\in \set{n+1-m_{xy}}\tq j\notin \theta(x)\} = \\
& = \#\theta(x)^c-(n+1-m_{xy}-\#\{j\in \set{n+1-m_{xy}}\tq j\in \theta(x)\}) = \\
& = \#\theta(x)^c-(n+1-m_{xy}-\eta_{\theta(x)}^{-1}(a))= m_{xy}+h-\#\theta(x)-1 = n-a+h-\#\theta(x).
\end{align*}

Thus, 
\begin{align*}
\tau(x)+\#\theta(x)-h-\tau(y)-q & = \sum\limits_{j\notin\theta(x)}(j+n)+ \#\theta(x)-h - q- \sum\limits_{j\notin\theta(y)}(j+n) = \\ &= -(a+n) + \#\theta(x)-h - q = \\
& = -a - n + \#\theta(x)-h - (n-a+h-\#\theta(x)) = 2\#\theta(x) - 2n -2h.
\end{align*}
Hence, $\tau(y)+q \equiv \tau(x)+\#\theta(x)-h \ \textnormal{(mod 2)}$.
\end{proof}

\begin{proposition} \label{prop_homology_boolean_lattice}
Let $L$ be a finite boolean lattice. Let $1_L=\max(L)$ and let $\F\colon  L-\{1_L\} \to \Ab$ be a covariant functor. Let $r=\rk(L)$. For each $n\in\Z$, let $D_n=\{x\in L-\{1_L\} \tq \rk(x)=r-n-1\}$. Let $\mathcal{C}=(C_n,d_n)_{n\in\Z}$ be the chain complex defined by $C_n=\bigoplus\limits_{x\in D_n}\F(x)$ with differentials $d_n\colon C_n\to C_{n-1}$ defined by $d_n(\lambda)=\sum \varepsilon(x\prec y) \F(x\leq y)(\lambda)$ for $\lambda \in \F(x)$ and sum over the elements $y$ which cover $x$.

Then $H_n(L-\{1_L\};\F)\cong H_n(\mathcal{C})$ for all $n\in \Z$.
\end{proposition}

\begin{proof}
Observe that $(L-\{1_L\})^\op$ is a quasicellular poset with quasicellular morphism $\rho\colon (L-\{1_L\})^\op \to \N_0$ defined by $\rho(x)=r-1-\rk(x)$. For each $n\in\Z$, let $D_n=\{x\in L-\{1_L\} \tq \rk(x)=r-n-1\}=\rho^{-1}(\{n\}\cap\N_0)$.

Applying theorem \ref{theo_homology_quasicellular_opposite} we obtain that $H_n(L-\{1_L\};\F)\cong H_n(\mathcal{C}^{L-\{1_L\},\F})$ for all $n\in \Z$, where 
$\mathcal{C}^{L-\{1_L\},\F}_m = \bigoplus\limits_{x\in D_m}\widetilde H_{m-1}(\widehat F_x,\F(x))$ for all $m\in \Z$ and with differentials defined as in theorem \ref{theo_homology_quasicellular_opposite}. 

For each $x\in L$ let $\upsilon_x\colon \F(x)\to \F(x)\otimes\Z$ be the isomorphism defined by $\upsilon_x(\lambda)=\lambda\otimes \varsigma(x)$.
Let $g_x\colon \Z\to H_m(U_{\theta(x)},\widehat U_{\theta(x)})$ be the isomorphism defined by $g_x(1)=\g_{\theta(x)}$.

Now note that, by the universal coefficient theorem and lemma \ref{lemma_generators}, for each $m\in\Z$ and for each $x\in D_m$ there is an isomorphism $\upsilon'_x\colon \F(x)\to \widetilde H_{m-1}(\widehat F_x,\F(x))$ given by $\upsilon'_x(\lambda)=\varsigma(x)\lambda\ldotp\theta_\ast ^{-1}(\partial_x([\g_{\theta(x)}]))$ where $\partial_x\colon H_{m-1}(U_{\theta(x)},\widehat U_{\theta(x)};\Z)\to \widetilde H_{m-2}(\widehat U_{\theta(x)})$ is the isomorphism obtained from the long exact sequence of the pair $(U_{\theta(x)},\widehat U_{\theta(x)})$.

It remains to prove that under this isomorphisms the differentials are given by the formula stated in the proposition. Concretely, we have to prove that for any cover relation $x\prec y$ in $L-\{1_L\}$, the following diagram commutes
\begin{displaymath}
\xymatrix{\F(x) \ar[r]^-{\upsilon_x} \ar[d]_{\varepsilon(x\prec y).\F(x\leq y)} & \widetilde H_{m-1}(\widehat F_x,\F(x)) \ar[d]^{\delta_{xy}} \\
\F(y) \ar[r]^-{\upsilon_y} & \widetilde H_{m-2}(\widehat F_y,\F(y))}
\end{displaymath}
where $m\in \Z$ is such that $x\in D_m$ and where $\delta_{xy}$ is the morphism induced by the differential $d_n^{L-\{1_L\};\F|}$ of theorem \ref{theo_homology_quasicellular_opposite} in the corresponding direct summands of the domain and codomain. Observe that the morphism $\delta_{xy}$ is defined by $\delta_{xy}\left(\left[\sum\limits_{i=1}^{l_x}a^x_i s^x_i\right]\right)=\left[\sum\limits_{s^x_i\ni y}\F(x\leq y)(a^x_i)\ldotp(s^x_i-\{y\})\right]$ where $l_x\in \N$, and for every $i=\{1,\dots,l_x\}$, $a^x_i\in \F(x)$ and $s^x_i\in \ch_{m-1}(\widehat{F}_x)$.

Since $x\prec y$, we have that $\theta(x)\supseteq \theta(y)$ and $\#(\theta(x)-\theta(y))=1$. Let $d\colon  H_m(U_{\theta(x)},\widehat U_{\theta(x)}) \to  H_{m-1}(U_{\theta(y)},\widehat U_{\theta(y)})$ be the map defined in lemma \ref{lemma_generators}.

Let $\widetilde\delta_{xy}\colon H_m(F_x,\widehat F_x,\F(x)) \to H_{m-1}(F_y,\widehat F_y,\F(y))$ be defined by 
\begin{displaymath}
\widetilde\delta_{xy}\left(\left[\sum\limits_{i=1}^{l_x}a^x_i s^x_i\right]\right)=\left[\sum\limits_{s^x_i\ni y}\F(x\leq y)(a^x_i)\ldotp(s^x_i-\{x\})\right]
\end{displaymath}
where $l_x\in \N$, and for every $i=\{1,\dots,l_x\}$, $a^x_i\in \F(x)$ and $s^x_i\in \ch_{m-1}(F_x,\widehat{F}_x)$.

Let $\beta=(-1)^{\#\theta(x)-h}\ldotp\id_\Z$.

Consider the following diagram, where the maps $\partial_x$, $\partial_y$, $\partial'_x$ and $\partial'_y$, are connection morphisms of the long exact sequences associated to the corresponding pairs, and where the maps labelled with $\cong$ are isomorphisms.
\begin{displaymath}
\xymatrix{ & H_m(U_{\theta(x)},\widehat U_{\theta(x)})\otimes\F(x) \ar@{<->}[r]^{UCT}_-{\cong} \ar@/^3pc/[ddd]^{d\otimes \F(x\leq y)} & H_m(U_{\theta(x)},\widehat U_{\theta(x)},\F(x)) \ar[r]^{\partial'_x}_-{\cong}  & \widetilde H_{m-1}(\widehat U_{\theta(x)},\F(x))
\\
\F(x) \ar[r]^-{\upsilon_x}_-{\cong} \ar[d]_{\varepsilon(x\prec y).\F(x\leq y)} & \Z\otimes\F(x) \ar[d]_{\beta \otimes \F(x\leq y)} \ar[u]^{g_x\otimes \id_{\F(x)}}_-{\cong} & H_m(F_x,\widehat F_x,\F(x)) \ar[r]^{\partial_x}_-{\cong} \ar[u]_{\theta_\ast}^-{\cong} \ar[d]^{\widetilde\delta_{xy}} & \widetilde H_{m-1}(\widehat F_x,\F(x)) \ar[d]^{\delta_{xy}} \ar[u]_{\theta_\ast}^-{\cong}
\\
\F(y) \ar[r]^-{\upsilon_y}_-{\cong} & \Z\otimes\F(y) \ar[d]_{g_y\otimes \id_{\F(y)}}^-{\cong} & H_{m-1}(F_y,\widehat F_y,\F(y)) \ar[r]^{\partial_y}_-{\cong}  \ar[d]^{\theta_\ast}_-{\cong} &  \widetilde H_{m-2}(\widehat F_y,\F(y)) \ar[d]_{\theta_\ast}^-{\cong}
\\
& H_m(U_{\theta(y)},\widehat U_{\theta(y)})\otimes\F(y) \ar@{<->}[r]_{UCT}^-{\cong} & H_m(U_{\theta(y)},\widehat U_{\theta(y)},\F(y)) \ar[r]^{\partial'_y}_-{\cong} & \widetilde H_{m-1}(\widehat U_{\theta(y)},\F(y)) 
}
\end{displaymath}

The left square is commutative by lemma \ref{lemma_sign_comparison}. The other part of the diagram is easily seen to be commutative. The result follows.
\end{proof}

\begin{theo} \label{theo_homology_boolean_lattice}
Let $L$ be a non-empty finite boolean lattice, let $1=\max(L)$ and let $r=\rk(L)$. For each $n\in\Z$, let $D_n=\{x\in L \tq \rk(x)=r-n\}$. Let $\F\colon L\to\Ab$ be a functor. Let $\mathcal{E}=(E_\ast,d^\mathcal{E}_\ast)$ be the chain complex defined by $E_n=\bigoplus_{x\in D_n}\F(x)$ for $n\in\N_0$, $E_n=0$ for $n<0$, with differentials $d^\mathcal{E}_n\colon E_n\to E_{n-1}$ defined by $d^\mathcal{E}_n(\lambda)=\sum \varepsilon(x\prec y) \F(x\leq y)(\lambda)$ for $\lambda \in \F(x)$ and sum over the elements $y$ which cover $x$.

Then $H_n(L,L-\{1\};\F)\cong H_n(\mathcal{E})$ for all $n\in \Z$.
\end{theo}

\begin{proof}
Let $\mathcal{C}=(C_\ast,d_\ast)$ be the chain complex of proposition \ref{prop_homology_boolean_lattice}. Note that $E_n=C_{n-1}$ and $d^\mathcal{E}_n=d_{n-1}$ for $n\in \N$. Thus, for $n\geq 2$ we obtain that 
$$H_n(\mathcal{E})\cong H_{n-1}(\mathcal{C}) \cong H_{n-1}(L-\{1\};\F|) \cong H_{n}(L,L-\{1\};\F)$$
where the second isomorphism follows from \ref{prop_homology_boolean_lattice} and the third isomorphism follows from \ref{prop_hom_Ux}.

It is not difficult to prove 
that the inclusion $\inc\colon C_0\to C_0(L-\{1\},\F|)$ induces an isomorphism $\inc_\ast\colon H_0(\mathcal{C})\to H_0(L-\{1\},\F|)$. Let $\alpha\colon \colim \F|_{L-\{1\}} \to \F(1)$ be the morphism induced by the maps $\F(y\to 1)$ for $y<1$ as in proposition \ref{prop_hom_Ux}. It is easy to verify that there is a commutative diagram
\begin{displaymath}
\xymatrix{C_0 \ar[rr]^{d^\mathcal{E}_1} \ar[d]_{q} & & \F(1) \\ H_0(\mathcal{C}) \ar[r]^-{\inc_\ast}_-\cong & H_0(L-\{1\};\F|_{L-\{1\}}) \ar[r]_-\cong & \colim \F|_{L-\{1\}} \ar[u]_\alpha}
\end{displaymath}
where $q$ is the quotient map.

Therefore,

$$H_0(\mathcal{E})=\coker d^\mathcal{E}_1 = \coker \alpha \cong H_{0}(L,L-\{1\};\F)$$

and
 
$$H_1(\mathcal{E})=\ker d^\mathcal{E}_1 / \Ima d^\mathcal{E}_2 = q(\ker d^\mathcal{E}_1) \cong \ker \alpha \cong H_{1}(L,L-\{1\};\F)$$

\end{proof}

The following corollary gives an alternative proof to a similar result of Everitt and Turner \cite[Theorem 24]{everitt2009homology}. In addition, our result gives a precise description of Khovanov's homology of knots as a particular instance of homology of posets with functor coefficients. Thus, we obtain a more conceptual proof of this fact.

\begin{coro} \label{coro_everitt_turner}
Let $L$ be a non-empty finite boolean lattice, let $1=\max(L)$ and let $r=\rk(L)$. Let $\F\colon L\to\Ab$ be a functor. Let $\mathcal(K)_\ast(L,\F)$ be the corresponding Khovanov's cube complex (as in \cite{everitt2009homology}). Then 
$$H_n(\mathcal(K)_\ast(L,\F))\cong H_n(L,L-\{1\};\F)$$
for all $n\in\Z$.
\end{coro}

\begin{proof}
Follows from theorem \ref{theo_homology_boolean_lattice}.
\end{proof}

\begin{coro}
Let $\mathcal{L}$ be a link diagram with $r$ crossings and let $KH_\ast(\mathcal{L})$ denote its Khovanov's homology. Let $B_{\mathcal{L}}$ be the associated boolean lattice of rank $r$ and let $\F_{KH}\colon B_{\mathcal{L}}\to\Ab$ be the functor corresponding to Khovanov's cube construction. Let $1_{B_{\mathcal{L}}}=\max B_{\mathcal{L}}$. Then $KH_n(\mathcal{L})\cong H_{r-n}(B_{\mathcal{L}},B_{\mathcal{L}}-\{1_{B_{\mathcal{L}}}\};\F_{KH})$
\end{coro}

\begin{proof}
Follows from corollary \ref{coro_everitt_turner}.
\end{proof}

\begin{notation}
Let $X$ be a poset and let $\sigma=\{\sigma_0,\dots,\sigma_n\}\in \ch_{n}(X)$. Let $x\in\sigma$. We will write $\sigma_x$ for the chain $\sigma-\{x\}$.
  
Furthermore, suppose that $x$ is an up beat point of $X$ and $y=\min \widehat{F}_x$ or that $x$ is a down beat point of $X$ and $y=\max \widehat{U}_x$. Then $\sigma\cup\{y\}$ is a chain in $X$ and will be denoted by $\sigma^{y}$. In particular, if $y\not\in\sigma$, then $\sigma^{y}\in\ch_{n+1}(X)$. We will also write $\sigma_x^{y}$ for $(\sigma_x)^{y}=(\sigma^{y})_x$.
  
For $0\leq k\leq n$, the chain $\sigma_{\sigma_k}=\sigma-\{\sigma_k\}$ will be denoted by $\sigma_{\widehat{k}}$.
\end{notation}

\begin{proposition} \label{prop_remove_ubp}
Let $P$ be a finite poset, let $a$ be an up beat point of $P$ and let $\F\colon P\to \Ab$ be a functor. Then $H_n(P;\F)\cong H_n(P-\{a\};\F|_{P-\{a\}})$ for every $n\in\N_0$.
\end{proposition}

\begin{proof}
Let $b=\min \widehat F_a$ and let $\varphi=\F(a<b)$. For every $n\in\N_0$ let $i_n\colon C_n(P-\{a\};\F|_{P-\{a\}})\to C_n(P;\F)$ be the group homomorphism induced by the inclusion $P-\{a\}\hookrightarrow P$ and let 
$r_n\colon C_n(P;\F)\to C_n(P-\{a\};F|_{P-\{a\}})$ be the group homomorphism defined by 
$$r_n(g\sigma)=\left\{
\begin{array}{lr}
g\sigma & \text{if $a\not\in\sigma$,}\\
0 & \text{if $\{a,b\}\subseteq \sigma$,}\\
g\sigma_a^{b} & \text{if $a\in\sigma$, $a\neq\min \sigma$ and $b\not\in\sigma$,}\\
\varphi(g)\sigma_a^{b} & \text{if $a\in\sigma$, $a=\min \sigma$ and $b\not\in\sigma$,}
\end{array}\right.$$
for $\sigma\in\ch_{n}(P)$ and $g\in \F(\min \sigma)$.
It is easy to see that $r_n$ is well defined and that $r_n$ is a retraction of $i_n$ for every $n\in\N_0$. Hence $r=\{r_n\}\colon C(P;\F)\to C(P-\{a\};\F|_{P-\{a\}})$ is a retraction of $i=\{i_n\}\colon  C(P-\{a\};\F|_{P-\{a\}})\to C(P;A)$.

For every $n\in\N_0$ let $\phi_n\colon C_n(P;\F)\to C_{n+1}(P;\F)$ be the group homomorphism defined by
$$\phi_n(g\sigma)=\left\{
\begin{array}{lr}
0 & \text{if $a\not\in\sigma$ or $\{a,b\}\subseteq \sigma$,}\\
-\sgn_\sigma(a) g \sigma^{b} & \text{if $a\in\sigma$ and $b\not\in\sigma$,}
\end{array}\right.$$
for $\sigma\in\ch_{n}(P)$ and $g\in \F(\min \sigma)$.

Let $\phi=\{\phi_n\}$.
We will show that $\phi d+d\phi=\id-ir$, that is, $\phi_{n-1}d_n+d_{n+1}\phi_n=\id_{C_n(P;\F)}-i_n r_n$ for every $n\in\N_0$. It will follow that $\phi$ is a chain homotopy between $\id_{C(P;\F)}$ and $ir$.

In what follows, we drop the subscript notation for the morphisms $\phi_n$, $d_n$, $i_n$, $r_n$ and $\id_{C_n(P;\F)}$. This should not create any confusion.

Let $n\in\N_0$, let $\sigma=\{\sigma_0,\dots,\sigma_n\}\in\ch_n(P)$ and let $g\in\F(\sigma_0)$. If $a\not\in\sigma$, then it is clear that $(\phi d+d\phi)(g\sigma)=0=(\id-ir)(g\sigma)$.
If $a\in\sigma$ and $b\in\sigma$ then $d\phi(g\sigma)=0$ and 
\begin{align*}
\phi d(g\sigma)&=\phi\left(\F(\sigma_0<\sigma_1)(g)\sigma_{\widehat{0}}+\sum\limits_{i=1}^{n}(-1)^{i}g\sigma_{\widehat{i}}\right)=\phi(\sgn_\sigma(b)g\sigma_b)=-\sgn_\sigma(a)\sgn_\sigma(b)g\sigma=g\sigma
\end{align*}
since $a$ and $b$ are consecutive elements of $\sigma$. Hence $(\phi d+d\phi)(g\sigma)=g\sigma=(\id-ir)(g\sigma)$.

Now suppose that $a=\sigma_0$ and $b\not\in\sigma$. In this case we have that 
\begin{align*}
d\phi(g\sigma)& =d(-g\sigma^{b})=-\F((\sigma^{b})_0<(\sigma^{b})_1)(g)(\sigma^{b})_{\widehat{0}}+\sum\limits_{i=1}^{n+1}(-1)^{i+1}g(\sigma^{b})_{\widehat{i}}= \\
&=-\varphi(g)\sigma^{b}_a+g\sigma+\sum\limits_{i=2}^{n+1}(-1)^{i+1}g(\sigma^{b})_{\widehat{i}})= -\varphi(g)\sigma^{b}_a+g\sigma+\sum\limits_{i=1}^{n}(-1)^{i}g(\sigma_{\widehat{i}})^{b}
\end{align*}
since $(\sigma^{b})_{\widehat{1}}=\sigma$ and $(\sigma^{b})_{\widehat{i+1}}=(\sigma_{\widehat{i}})^{b}$ for $1\leq i\leq n$.

On the other hand, we have that 
\begin{align*}
\phi d(g\sigma)=\phi\left(\F(\sigma_0<\sigma_1)(g)\sigma_{\widehat{0}}+\sum\limits_{i=1}^{n}(-1)^{i}g\sigma_{\widehat{i}}\right)=\sum\limits_{i=1}^{n}(-1)^{i}\phi(g\sigma_{\widehat{i}})=\sum\limits_{i=1}^{n}(-1)^{i+1}g(\sigma_{\widehat{i}})^{b}.
\end{align*}
It follows that $(d\phi+\phi d)(g\sigma)=g\sigma-\varphi(g)\sigma_a^{b}=(\id-ir)(g\sigma)$.

Finally, suppose that $a=\sigma_k$ for $1\leq k\leq n$ and that $b\not\in\sigma$. Then
$$d\phi(g\sigma)=(-1)^{k+1}\F((\sigma^{b})_0<(\sigma^{b})_1)(g)(\sigma^{b})_{\widehat{0}}+\sum\limits_{i=1}^{n+1}(-1)^{i+k+1}g(\sigma^{b})_{\widehat{i}}$$
and
$$d\phi(g\sigma)=(-1)^{k}\F(\sigma_0<\sigma_1)(g)(\sigma_{\widehat{0}})^{b}+\sum\limits_{i=1}^{k-1}(-1)^{i+k}g(\sigma_{\widehat{i}})^{b}-\sum\limits_{i=k+1}^{n}(-1)^{i+k}g(\sigma_{\widehat{i}})^{b}.$$
Now, observe that $(\sigma_{\widehat{i}})^{b}=(\sigma^{b})_{\widehat{i}}$ for $0\leq i\leq k-1$, $(\sigma_{\widehat{i}})^{b}=(\sigma^{b})_{\widehat{i+1}}$ for $k+1\leq i\leq n$, $(\sigma^{b})_{\widehat{k}}=\sigma_a^{b}$ and $(\sigma^{b})_{\widehat{k+1}}=\sigma$. It follows that $(d\phi+\phi d)(g\sigma)=g\sigma-g\sigma_a^{b}=(\id-ir)(g\sigma)$.

Thus $\id$ and $ri$ are chain homotopic, which shows that $C(P;\F)$ and $C(P-\{a\};\F|_{P-\{a\}})$ are homotopy equivalent chain complexes. Therefore $H_n(P;F)\cong H_n(P-\{a\};\F|_{P-\{a\}})$ for every $n\in\N_0$.
\end{proof}

\begin{proposition} \label{prop_remove_dbp}
Let $P$ be a finite poset, let $a$ be a down beat point of $P$ and let $b=\max \widehat F_a$. Let $\F\colon P\to \Ab$ be a functor. If $\F(b\leq a)$ is an isomorphism then $H_n(P;\F)\cong H_n(P-\{a\};\F)$ for every $n\in\N_0$.
\end{proposition}

\begin{proof}
Let $\varphi=\F(b<a)^{-1}$. Let $i_n$ and $r_n$ be defined as in the proof of \ref{prop_remove_ubp} for every $n\in\N_0$, as well as $i$ and $r$. Again, it is easy to see that $r_n$ is well defined for every $n\in\N_0$.

As before, our purpose is to define a chain homotopy between $\id$ and $ir$. For every $n\in\N_0$ let $\phi_n\colon C_n(P;\F)\to C_{n+1}(P;\F)$ be the group homomorphism defined by
$$\phi_n(g\sigma)=\left\{
\begin{array}{lr}
0 & \text{if $a\not\in\sigma$ or $\{a,b\}\subseteq \sigma$,}\\
\sgn_\sigma(a) g \sigma^{b} & \text{if $a\in\sigma$, $a\neq\min\sigma$ and $b\not\in\sigma$,}\\
\varphi(g)\sigma^{b} & \text{if $a\in\sigma$, $a=\min\sigma$ and $b\not\in\sigma$,}
\end{array}\right.$$
for $\sigma\in\ch_{n}(P)$ and $g\in \F(\min \sigma)$ and let $\phi=\{\phi_n\}$.

Let $n\in\N_0$, let $\sigma=\{\sigma_0,\dots,\sigma_n\}\in\ch_n(P)$ and let $g\in\F(\sigma_0)$. If $a\not\in\sigma$ it is immediate that $(\phi d+d\phi)(g\sigma)=0=(\id-ir)(g\sigma)$.

Suppose that $\{a,b\}\subseteq \sigma$. Then $d\phi(g\sigma)=0$ and

$$\phi d(g\sigma)=\phi\left(\F(\sigma_0<\sigma_1)(g)\sigma_{\widehat 0}+\sum\limits_{i=1}^{n}(-1)^{i}g\sigma_{\widehat i}\right)=
\left\{\begin{array}{lr}
\phi(\varphi^{-1}(g)\sigma_b)&\text{if $\sigma_0=b$,}\\
\phi(\sgn_\sigma (b) g \sigma_b)&\text{if $\sigma_0\neq b$.}
\end{array}\right.
$$
Since $\sgn_\sigma(b)=\sgn_{\sigma_b}(a)$, it follows from the definition of $\phi$ that $\phi d(g\sigma)=g\sigma$ whether or not $\sigma_0=b$, and thus $(d\phi+\phi d)(g\sigma)=g\sigma=(\id-ir)(g\sigma)$.

Now suppose that $a=\sigma_0$ and $b\not\in\sigma$. Then
$$d\phi(g\sigma)=d(\varphi(g)\sigma^{b})=g\sigma+\sum\limits_{i=1}^{n+1}(-1)^{i}\varphi(g)(\sigma^{b})_{\widehat{i}}=g\sigma-\varphi(g)\sigma_a^{b}+\sum\limits_{i=2}^{n+1}(-1)^{i}\varphi(g)(\sigma^{b})_{\widehat{i}}$$
and
$$\phi d(g\sigma)=\sum\limits_{i=1}^{n}(-1)^{i}\varphi(g)(\sigma_{\widehat{i}})^{b}.$$

Since $(\sigma_{\widehat{i}})^{b}=(\sigma^{b})_{\widehat{i+1}}$ for $1\leq i\leq n$, then
$(d\phi+\phi d)(g\sigma)=g\sigma-\varphi(g)\sigma_a^{b}=(\id-ir)(g\sigma)$.

Finally, suppose that $a=\sigma_k$ with $k\geq 1$ and $b\not\in\sigma$. If $k=1$, then 
$$\phi(\F(\sigma_0<\sigma_1)(g)\sigma_{\widehat{0}})=(\varphi\circ\F(\sigma_0<a))(g)(\sigma_{\widehat{0}})^{b}=\F(\sigma_0<b)(g)(\sigma_{\widehat{0}})^{b}$$
and if $k\geq 2$ then
$$\phi(\F(\sigma_0<\sigma_1)(g)\sigma_{\widehat{0}})=\sgn_{\sigma_{\widehat{0}}}(a)\F(\sigma_0<\sigma_1)(g)(\sigma_{\widehat{0}})^{b}.$$
In either case, it is clear that 
$$\phi(\F(\sigma_0<\sigma_1)(g)\sigma_{\widehat{0}})=-\sgn_\sigma(a)\F((\sigma^{b})_0<(\sigma^{b})_1)(g)(\sigma_{\widehat{0}})^{b}.$$
Thus
$$\phi d(g\sigma)=-\sgn_\sigma(a)\F((\sigma^{b})_0<(\sigma^{b})_1)(g)(\sigma_{\widehat{0}})^{b}+\sum\limits_{i=1}^{n}(-1)^{i}\sgn_{\sigma_{\widehat{i}}}(a)g(\sigma_{\widehat{i}})^{b}.$$

On the other hand,
$$d\phi(g\sigma)=\sgn_\sigma(a)\F((\sigma^{b})_0<(\sigma^{b})_1)(g)(\sigma^{b})_{\widehat{0}}+\sgn_\sigma(a)\sum\limits_{i=1}^{n+1}(-1)^{i}(\sigma^{b})_{\widehat{i}}.$$
Using the fact that $(\sigma_{\widehat{i}})^{b}=(\sigma^{b})_{\widehat{i+j}}$ where $j=0$ if $1\leq i\leq k-1$ and $j=1$ if $k+1\leq i\leq n$, we obtain that
$$(d\phi+\phi d)(g\sigma)=g\sigma-g\sigma_a^{b}=(\id-ir)(g\sigma).$$  The result follows.
\end{proof}

Let $(P,\leq_P)$ and $(Q,\leq_Q)$ be finite posets and let $f\colon P\to Q$ be an order-preserving map. We define $P\cup_f Q$ as the poset whose underlying set is the disjoint union $P\sqcup Q$ and whose ordering $\leq$ is defined as follows. If $p_1,p_2\in P$ then $p_1\leq p_2$ if and only if $p_1\leq_P p_2$. If $q_1,q_2\in Q$ then $q_1\leq q_2$ if and only if $q_1\leq_Q q_2$. If $p\in P$ and $q\in Q$ then $p\leq q$ if and only if $f(p)\leq_Q q$.

Now, let $\F_P\colon P\to\Ab$ and $\F_Q\colon Q\to\Ab$ be covariant functors and let $\phi\colon \F_P\Rightarrow\F_Q\circ f$ be a natural transformation. We define the functor $\F_P\underset{f,\phi}{\cup} \F_Q\colon P\cup_f Q \to \Ab$ as in \cite{everitt2009homology}. This is, in objects, $(\F_P\underset{f,\phi}{\cup} \F_Q)(p)=\F_P(p)$ for $p\in P$ and $(\F_P\underset{f,\phi}{\cup} \F_Q)(q)=\F_Q(q)$ for $q\in Q$. In arrows, $(\F_P\underset{f,\phi}{\cup} \F_Q)(p_1\leq p_2)=\F_P(p_1\leq_P p_2)$ for $p_1,p_2\in P$, $(\F_P\underset{f,\phi}{\cup} \F_Q)(q_1\leq q_2)=\F_Q(q_1\leq_Q q_2)$ for $q_1,q_2\in Q$ and  $(\F_P\underset{f,\phi}{\cup} \F_Q)(p\leq q)=\F_Q(f(p)\leq_Q q)\circ\phi_p$ for $p\in P$ and $q\in Q$ .

\begin{proposition} \label{prop_mapping_cylinder}
Let $P$ and $Q$ be finite posets and let $f\colon P\to Q$ be an order-preserving map. Let $\F_P\colon P\to\Ab$ and $\F_Q\colon Q\to\Ab$ be covariant functors and let $\phi\colon \F_P\Rightarrow\F_Q\circ f$ be a natural transformation. Then the inclusion map $Q\hookrightarrow P\cup_f Q$ induces an isomorphism  $H_n(P\cup_f Q;\F_P\underset{f,\phi}{\cup} \F_Q)\cong H_n(Q;\F_Q)$ for all $n\in\Z$.
\end{proposition}

\begin{proof}
Follows from \ref{prop_remove_ubp} since the poset $Q$ can be obtained from the poset $P\cup_f Q$ by successively removing up beat points (note that the maximal points of $P$ are up beat points of $P\cup_f Q$).
\end{proof}

\begin{theo}
Let $P$ and $Q$ be finite posets such that both of them have a maximum. Let $1_P=\max P$ and let $1_Q=\max Q$. Let $f\colon P\to Q$ be an order-preserving map such that $f^{-1}(\{1_Q\})=\{1_P\}$. Let $\F_P\colon P\to\Ab$ and $\F_Q\colon Q\to\Ab$ be covariant functors and let $\phi\colon \F_P\Rightarrow\F_Q\circ f$ be a natural transformation.

Then there exists a long exact sequence
\begin{align*}
\cdots &\longrightarrow H_n(Q,Q-\{1_Q\};\F_Q) \longrightarrow H_n(P\cup_f Q,P\cup_f Q-\{1_Q\};\F_P\underset{f,\phi}{\cup} \F_Q) \longrightarrow H_{n-1}(P,P-\{1_P\};\F_P) \longrightarrow \\ 
& \longrightarrow H_{n-1}(Q,Q-\{1_Q\};\F_Q) \longrightarrow \cdots
\end{align*}
\end{theo}

\begin{proof}
Let $X=P\cup_f Q$ and let $\F_X=\F_P\underset{f,\phi}{\cup} \F_Q$.
There exists a long exact sequence
\begin{displaymath}
\cdots \longrightarrow H_n(X-\{1_Q\},X-\{1_Q,1_P\};\F_X|_{X-\{1_P\}}) \longrightarrow H_n(X,X-\{1_Q,1_P\};\F_X) \longrightarrow H_{n}(X,X-\{1_Q\};\F_X) \longrightarrow \cdots
\end{displaymath}
Observe that, applying \ref{prop_hom_X_A_Cx}, we obtain that 
$$H_{n}(X-\{1_Q\},X-\{1_Q,1_P\};\F_X|_{X-\{1_Q\}})\cong H_{n}(U_{1_P},\widehat U_{1_P};\F_X|_{U_{1_P}}) \cong H_{n}(P,P-\{1_P\};\F_P)$$
for all $n\in\Z$.

On the other hand, from \ref{prop_mapping_cylinder} we obtain that the inclusion maps $Q\hookrightarrow X$ and $Q-\{1_Q\}\hookrightarrow X-\{1_Q,1_P\}$ induce isomorphisms $H_n(Q;\F_Q) \cong H_n(X;\F_X)$ and $H_n(Q-\{1_Q\};\F_Q|_{Q-\{1_Q\}}) \cong H_n(X-\{1_Q,1_P\};\F_X|_{X-\{1_Q,1_P\}})$ for all $n\in\Z$. Therefore, from the long exact sequences in homology groups with coefficients in the functor $\F_X$ (or its corresponding restrictions) for the pairs $(Q,Q-\{1_Q\})$ and $(X,X-\{1_Q,1_P\})$, and applying naturality (see page \pageref{morphism_induced_by_natural_transformation}) and the five lemma we obtain that $H_n(X,X-\{1_Q,1_P\};\F_X)\cong H_n(Q,Q-\{1_Q\};\F_Q)$ for all $n\in\Z$.

Therefore, the long exact sequence of the beginning of this proof turns out to be the long exact sequence of the statement of the theorem.
\end{proof}

As a corollary, we obtain an alternative and more conceptual proof to a known result of Khovanov homology which was also proved by Everitt and Turner in \cite{everitt2009homology}.

\begin{coro}
Let $B$ be a boolean lattice. Let $B_0$ and $B_1$ as in Everitt-Turner. Let $\F\colon B\to\Ab$ be a functor. Then there exists a long exact sequence
\[\cdots \longrightarrow H_n(B_1;\F|) \longrightarrow H_n(B;\F) \longrightarrow H_{n-1}(B_0;\F|) \longrightarrow \cdots \]
\end{coro}

\bibliographystyle{acm}
\bibliography{ref_hom_funct_coeff}

\end{document}